\documentclass{amsart}
\usepackage{packages}
\usepackage{xcolor}
\usepackage{commath}
\newcommand{\R}{\mathbb{R}}

\usepackage[square, numbers, sort]{natbib}

\title{Lie minimal Weingarten surfaces}

\author{Joseph Cho}
\address[Joseph Cho]{Institute of Discrete Mathematics and Geometry, TU Wien, Wiedner Hauptstrasse 8-10/104, 1040 Wien, Austria}
\email{jcho@geometrie.tuwien.ac.at}

\author{Masaya Hara}
\address[Masaya Hara]{Department of Mathematics, Graduate School of Science, Kobe university, 1-1 Rokkodai-cho, Nada-ku, Kobe 657-8501, Japan}
\email{mhara@math.kobe-u.ac.jp}

\author{Denis Polly}
\address[Denis Polly]{Institute of Discrete Mathematics and Geometry, TU Wien, Wiedner Hauptstrasse 8-10/104, 1040 Wien, Austria}
\email{dpolly@geometrie.tuwien.ac.at}

\author{Tomohiro Tada}
\address[Tomohiro Tada]{Department of Mathematics, Graduate School of Science, Kobe university, 1-1 Rokkodai-cho, Nada-ku, Kobe 657-8501, Japan}
\email{214s006s@gsuite.kobe-u.ac.jp}

\date{v4 \today}


\begin{document}

\begin{abstract}
    We consider Lie minimal surfaces, the critical points of the simplest Lie sphere invariant energy, in Riemannian space forms. These surfaces can be characterized via their Euler-Lagrange equations, which take the form of differential equations of the principal curvatures. Surfaces with constant mean curvature that satisfy these equations turn out to be rotational in their space form. We generalize in flat ambient space: here surfaces where the principal curvatures satisfy an affine relationship as well as elliptic linear Weingarten surfaces are rotational as well. 
\end{abstract}

\maketitle

\section{Introduction}
Willmore functional was first considered by Germain \cite{germain_recherches_1821} as a surface analogue of the bending energy of curves whose minimizers are the elastic curves of Bernoulli and Euler.
The critical points of the Willmore functional under compactly supported variations are called \emph{constrained Willmore surfaces}, and they have been widely studied (see, for example, \cite{bohle_constrained_2008, burstall_schwarzian_2002}), in part, due to the now-resolved Willmore conjecture \cite{marques_min-max_2014, willmore_note_1965}.
In particular, the Willmore functional is conformally invariant, and is in fact the simplest variational problem that can be considered in conformal geometry, or Möbius geometry \cite{thomsen_uber_1924}.
On the other hand, those surfaces not determined by invariants of surfaces in a geometry are called \emph{deformable surfaces}, and Cartan has showed that \emph{isothermic surfaces}, those surfaces that admit conformal curvature lines coordinates, are exactly the deformable surfaces in Möbius geometry \cite{cartan_les_1923}.
Interestingly, constant mean curvature surfaces in space forms are important examples that are both constrained Willmore and isothermic \cite{thomsen_uber_1924,richter_conformal_1997}.
Thus, certain curvature properties, a space form notion, play an important role in consideration of both minimality and deformability of Möbius geometry.

Similar considerations can be made in other sphere geometries.
In the context of Laguerre geometry, the simplest variational problem revolves around a functional that is Laguerre invariant called \emph{Weingarten functional}; the critical points of Weingarten functional under compactly supported variations are called \emph{$L$-minimal surfaces} \cite{blaschke_uber_1924-1} (see also \cite{musso_variational_1996}).
On the other hand, the deformable surfaces in Laguerre geometry are exactly \emph{$L$-isothermic surfaces} \cite{blaschke_uber_1924}, those surfaces with curvature line coordinates that are conformal with respect to the third fundamental form.
Again, certain curvature properties play a crucial role in determining the surfaces with both minimality and deformability, as minimal surfaces in Euclidean space are both $L$-minimal and $L$-isothermic \cite{blaschke_uber_1924-1,musso_variational_1996}.

In this paper, we focus on Lie sphere geometry \cite{lie_ueber_1872-1}, the sphere geometry that includes both Möbius geometry and Laguerre geometry as its subgeometries, and consider Lie minimal surfaces that satisfy certain curvature properties in space forms.

First considered by Blaschke \cite{blaschke1929}, \emph{Lie minimal surfaces} are surfaces that are the critical points of the simplest energy $L_{\textrm{Lie}}$, invariant under Lie sphere transformations, under compactly supported variations. This class of surfaces has raised interest recently, notably from two different perspectives: \cite{MR1756094} showed that they constitute an integrable system (an integrable reduction of the Gauss-Codazzi equations of the Lie sphere frame), while \cite{MR1918585} demonstrated that they have harmonic conformal Gauss maps (a congruence of Lie cyclides). 

These two characetrizations are purely Lie sphere geometric in nature. However, given a surface in a Riemannian space form, the energy $L_{\textrm{Lie}}$ can be computed using the principal curvatures of the space form projection \cite{MR1756094}. The goal of this paper is gain insights into the geometry of Lie minimal surfaces whose principal curvatures satisfy additional properties, namely, relationships of the Weingarten type. We will show that affine Weingarten and elliptic linear Weingarten surfaces are Lie minimal if and only if they are rotational in their space forms in the sense of \cite{MR694383}. 

In Section~\ref{sec:preliminaries} we will lay the groundwork for our investigations. Starting from the expression of $L_{\textrm{Lie}}$ via principal curvatures of the space form projection, we will develop the Euler-Lagrange equations for Lie minimal surfaces in space forms (Lemma~\ref{lem:lieMinimalConditions}). These differential equations for the principal curvatures are then used in Sections~\ref{sec:lieMinimalcmc} and \ref{sec:lieMinimalWeingarten} to investigate Lie minimal surfaces with additional curvature conditions. 

As minimal surfaces in Euclidean space are both constrained Willmore and L-minimal, while cmc surfaces in space forms are constrained Willmore, the starting point of our investigations are minimal Lie minimal surfaces or, more generally, Lie minimal surfaces with constant mean curvature ${H=\tfrac{1}{2}(k_1 + k_2)}$. In contrary to the Möbius or Laguerre geometric cases, we show that these surfaces must be rotational within their space form in the sense of \cite{MR694383} (Theorem~\ref{Theorem: cmc Lie minimal}). Since rotational cmc surfaces in Riemannian space forms have been completely classified (most recently in \cite{polly2023rotational}), this provides a complete classification of all Lie minimal cmc surfaces.  

Following this, we consider two more classes of surfaces with constraint principal curvatures in Section~\ref{sec:lieMinimalWeingarten}: on the one hand, we consider surfaces in $\R^3$ where the principal curvatures satisfy an affine linear relationship of the form
\[
 xk_1 + yk_2 = z.
\]
This extends the class of cmc surfaces, as these correspond to the special case $x=y$. We call this class the class of \emph{affine Weingarten surfaces}. We will demonstrate that Lie minimal affine Weingarten surfaces are always rotational in $\R^3$ (Theorem~\ref{thm:AffineLieMinimal}). The classification of rotational affine Weingarten surfaces of, e.g., \cite{MR4089078} thus provides a complete classification of this class as well. 

On the other hand, we consider the class of \emph{linear Weingarten} surfaces, examples of deformable surfaces in Lie sphere geometry \cite{burstall2012}.
This class extends the class of cmc surfaces to those where there is an affine linear relationship between the mean curvature $H$ and the Gauss curvature $K = k_1 k_2$, that is
\[
 aK + 2bH + c=0.
\]
According to Bonnet's theorem (Proposition~\ref{prop:Bonnet}), all parallel transformations of linear Weingarten surfaces are again linear Weingarten. Since the parallel transformation preserves rotational surfaces and Lie minimality (a Lie sphere invariant notion), we can prove that all Lie surfaces that are parallel to a cmc surface are rotational. In the case of Euclidean ambient geometry, this comprises the class of linear Weingarten surfaces with $b^2-ac>0$, dubbed elliptic linear Weingarten surfaces in \cite{lopez2008c}.

\section{Preliminaries}\label{sec:preliminaries}
In this section we describe the setup of our investigations. The main goal is to fix notation we will use for the theory of surfaces in space forms. In Subsection~\ref{sec:lieMinimalSurfaces} we will define the notion of Lie minimality. The main result of this section is Lemma~\ref{lem:lieMinimalConditions} which we use to prove that channel surfaces are always Lie minimal (and in later sections). 

Let $M_\kappa$ be a Riemannian space form with constant sectional curvature $\kappa$. We think of $M_\kappa$ as either $\R^3$ ($\kappa=0$) or a three dimensional (Riemannian) quadric in $\R^4$ ($\kappa=1$) or $\R^{3,1}$ ($\kappa=-1$). Let $X: \Sigma \rightarrow M_\kappa$ be an umbilic-free immersion with \emph{curvature line} coordinates $(u,v)$, that is, the first and second fundamental form read as 
\begin{equation}\label{eq:fundamentalForms2}
	\mathrm{I}(u,v)=E \dif{u}^2 + G \dif{v}^2 \quad \text{and} \quad \mathrm{II}(u,v)=\mathrm{L} \dif{u}^2 + \mathrm{N} \dif{v}^2. 
\end{equation}
We call the coordinates $(u,v)$ \emph{isothermic} if they are additionally conformal, that is, there exists a function $\sigma$ such that
\begin{equation}\label{eq:fundamentalForms}
	\mathrm{I}(u,v)=e^{2\sigma}(\dif{u}^2 + \dif{v}^2) \quad \text{and} \quad \mathrm{II}(u,v)=e^{2\sigma}(k_1 \dif{u}^2 + k_2 \dif{v}^2),
\end{equation}
where $k_1, k_2$ denote the principal curvatures. We use these to define the \emph{mean} and \emph{(extrinsic) Gauss curvature} by
\[
    H=\frac{k_1 + k_2}{2}, \quad K=k_1 k_2,
\]
respectively. As we are assuming umbilic-free, we have $k_1 \neq k_2$. 

We briefly review the structure equations for surfaces in space forms.
\begin{proposition}[{\cite[\S 65]{MR0115134}}]
	Let $X: \Sigma \rightarrow M_\kappa$ be a surface with curvature line coordinates $(u,v)$ and fundamental forms as in \eqref{eq:fundamentalForms2}.
    Then the Codazzi equations are
	\begin{equation}\label{The Codazzi equation for curv line coordinates}
		\frac{k_{1,v}}{k_2-k_1}=\big({\log{\sqrt{E}}}\big)_v \quad \text{and} \quad \frac{k_{2,u}}{k_1-k_2}=\big({\log{\sqrt{G}}}\big)_u.
	\end{equation}
    If $(u,v)$ are conformal coordinates in addition (so that they are isothermic coordinates) with fundamental forms as in \eqref{eq:fundamentalForms}, then the Gauss equation takes the form
	\begin{equation}\label{The Gauss equation for isothermic coordinates}
		\sigma_{uu}+\sigma_{vv}+(\kappa+K) e^{2\sigma}=0.
	\end{equation}
    while the Codazzi equations simplify to
    \begin{equation}\label{The Codazzi equation for isothermic coordinates}
		\frac{k_{1,v}}{k_2-k_1}=\big({\log{\sqrt{E}}}\big)_v = \sigma_v \quad \text{and} \quad \frac{k_{2,u}}{k_1-k_2}=\big({\log{\sqrt{G}}}\big)_u = \sigma_u.
	\end{equation}
\end{proposition}

As rotational surfaces in various space forms are central to our investigations, let us introduce this notion: consider $M_\kappa$ as a quadric in the appropriate ambient space, equipped with the proper inner product ($\R^3$ can be viewed as $\{p: (v,p)=1\}$ with $v\in\R^4$ any non-zero vector). Isometries of $M_\kappa$ are then those orthogonal transformations of the ambient space that fix the space form. We call a surface \emph{rotational} if it is invariant under a $1$-parameter subgroup $\rho$ of isometries that act as the identity on a $2$-plane\footnote{
        Note that for $\R^3$, the plane that is fixed by $\rho$ must contain $v$.} 
$\Pi$ in $\R^4$. For hyperbolic spaces so that $\kappa<0$, the definition amounts to three different types of rotations, depending on the signature of the metric induced on $\Pi$ (for an extensive treatment of rotations in hyperbolic space forms see \cite{MR694383}).
The next example gives explicit parametrizations of rotational surfaces along with their principal curvatures:
\begin{example}\label{exp:rotationalSurfaces}
    For brevity we assume that the induced metric on $\Pi^\perp$ is Riemannian. In this case, the term \emph{surface of revolution} is used more frequently. 
    We can parametrize rotational surfaces in $M_\kappa$ via the action of $\rho$ on a (geodesic) planar profile curve $\gamma = (r, 0, h, k)$ ($k\equiv 1$ for $\kappa=0$). From the  parametrization
    \begin{equation}\label{eq:SOR}
        X(u, v) = \rho(u)(r(v), 0, h(v), k(v)),
    \end{equation}
    and the fact that $\rho$ is a $1$-parameter subgroup, one quickly deduces
    that $\mathrm{I}$ and $\mathrm{II}$ take the form \eqref{eq:fundamentalForms}
    with (we use $^\prime$ to denote $\partial_v$)
    \[
        e^{2\sigma} = r^2,~k_1 = \frac{kh'-hk'}{r^2}, ~k_2 = \frac{1}{r^3} \begin{vmatrix}
            r &r' &r'' \\
            h &h' &h'' \\
            k &k' &\kappa k''
        \end{vmatrix},
    \]
    given a suitable parametrization of $\gamma$. Similar expressions for the other types of rotations in hyperbolic space can be obtained.
\end{example}

Note that for a rotational surface $k_{1,u}=0$. Thus, rotational surfaces are an example of channel surfaces as defined next. 

\begin{definition}
    A surface in a space form $M_\kappa$ is a \emph{channel surface} if one of its principal curvatures is constant along its principal direction. 
\end{definition}

\begin{remark}
    According to Joachimsthal's theorem, this definition is equivalent to the more usual definition of channel surfaces as envelopes of $1$-parameter families of spheres \cite{blaschke1929}.
\end{remark}

\subsection{Linear Weingarten surfaces}\label{sec:linearWeingartenSurfaces}
Next, we introduce the class of linear Weingarten surfaces. There are two different conventions which surfaces to call linear Weingarten: either those for which Gauss and mean curvature satisfy a linear relationship, or those where the two principal curvatures do. We call the latter ``affine Weingarten'' to distinguish these two surface classes.

\begin{definition}
    A surface in a space form $M_\kappa$ is called
    \begin{itemize}
        \item \emph{linear Weingarten} if there exists a non-trivial triple of constants $a,b,c$ such that 
        \begin{equation}\label{eq:LW}
            aK + 2bH + c = 0. \tag{LW}
        \end{equation}
        \item \emph{affine Weingarten} if there exists a non-trivial triple of constants $a,b,c$ such that
        \begin{equation}\label{eq:AW}
            xk_1 + yk_2 + z=0. \tag{AW}
        \end{equation}
    \end{itemize}
\end{definition}

\begin{remark}
    The only surfaces that are linear Weingarten and affine Weingarten are surfaces of constant mean curvature (cmc): set $a=0$ in \eqref{eq:LW} and $x=y$ in \eqref{eq:AW}. 
\end{remark}

We call a linear Weingarten surface $X:\Sigma\to M_\kappa$ \emph{tubular} if $ac-b^2 = 0$. This is equivalent to one of the principal curvatures of $X$ being constant. Therefore, tubular linear Weingarten surfaces are also affine Weingarten ($y=0$) and channel surfaces. We will from now on assume further that the surfaces we consider are non-tubular. 

Regarding channel linear or affine Weingarten surfaces, we have the following result:
\begin{proposition}\label{proposition:channelWeingartenSurfacesAreRotational} 
Let a surface be a non-tubular channel surface.
If we have either
    \begin{itemize}
        \item it is a linear Weingarten surface in a space form $M_\kappa$, or
        \item it is an affine Weingarten surface in $\R^3$,
    \end{itemize}
then the surface must be rotational.
\end{proposition}

\begin{proof}
    The first statement is proven in, for instance,  \cite{hertrichjeromin2021channel}. 

    For the second statement, note that by \eqref{eq:AW} $k_1=k_1(u)$ implies that $k_2$ is a function of $u$ alone as well. The Codazzi equations 
    \eqref{The Codazzi equation for isothermic coordinates} then imply $E=E(u)$ and $G=U(u)V(v)$ for suitable functions $U, V$ of one variable.
    Thus by \cite{MR1575524},
    we have that the surface admits isothermic coordinates.
    More precisely, we can set $r=\sqrt{E}$ and define a function $h$ (of $u$) such that $h_u=E k_2$. The Gauss-Codazzi equations then imply that the fundamental forms are
    \begin{align*}
        \mathrm{I}(u,v)= 
            r^2(\dif{u}^2 +\dif{v}^2)  \quad \text{and} \quad \mathrm{II}(u,v)=r^2k_1 \dif{u}^2 + h_u \dif{v}^2,
    \end{align*}
    with $k_1 = \tfrac{r_u^2 - r_{uu}r}{h_u r^2}$. A quick computation shows that the surface of revolution with $r$ and $h$ chosen such that \eqref{eq:SOR} is isothermic ($r^2 = r_u^2 + h_u^2$) has the same fundamental forms, and the claim thus follows from the fundamental theorem. 
\end{proof}

\subsection{Lie minimal surfaces}\label{sec:lieMinimalSurfaces}
Next, we define Lie minimal surfaces $X: \Sigma \rightarrow M_\kappa$.
\begin{definition}[\cite{blaschke1929, MR1918585, MR1756094, MR2186784}]\label{def:LieMin}
	The critical points of a functional
    \begin{align}\label{eq:lieMinimal}
	L_{Lie}[X]:=\int_\Sigma \frac{k_{1,u} k_{2,v}}{(k_1-k_2)^2} \dif{u} \wedge \dif{v}
    \end{align}
	with respect to compactly supported variations are called \emph{Lie minimal surfaces}.
\end{definition}

The functional \eqref{eq:lieMinimal} is introduced in \cite{blaschke1929} as the simplest Lie invariant integral. Thus, Lie sphere transformations map Lie minimal surfaces to Lie minimal surfaces. 

We now state a characterization of Lie minimal surfaces in terms of a pair of differential equations that the principal curvatures of $X$ have to satisfy. 

\begin{lemma}\label{lem:lieMinimalConditions}
	Let $X: \Sigma \rightarrow M_\kappa$ be a curvature line parametrization of a surface in a space form. Then $X$ is Lie minimal if and only if the principal curvatures satisfy
	\begin{equation}\label{eq:lieMinimalConditions}
		(k_2-k_1)k_{1,uv}+2k_{1,u}k_{1,v}=0 \quad \text{and} \quad (k_1-k_2)k_{2,uv}+2k_{2,u}k_{2,v}=0.
	\end{equation}
\end{lemma}

\begin{proof}
Consider a compactly supported variation of the surface $ X $, with variation parameter $ \varepsilon $ independent of $u$ and $v$, that fixes boundaries, and suppose that it results in, for the principal curvatures $ \hat{k}_1, \hat{k}_2 $ at $ \varepsilon $,
	\begin{align*}
		k_1(u,v) &\mapsto \hat{k}_1(u,v) = k_1(u,v) + \varepsilon h_1(u,v) + \mathcal{O}(\varepsilon^2) \\
		k_2(u,v) &\mapsto \hat{k}_2(u,v) = k_2(u,v) + \varepsilon h_2(u,v) + \mathcal{O}(\varepsilon^2)
	\end{align*}
for suitably chosen functions $ h_1,h_2:\Sigma \rightarrow \mathbb{R} $ such that, for $ i = 1,2 $, $ h_i(u,v)=0 $ outside of the compact support of the variation. Setting
	\begin{align*}
		f = f\lbrack k_1, k_2, k_{1,u}, k_{2,v} \rbrack = \frac{k_{1,u} k_{2,v}}{(k_1 - k_2)^2},
	\end{align*}
then
	\begin{multline*}
		\left.\frac{\dif{L}}{\dif{\varepsilon}} \right|_{\varepsilon = 0}
            = \left. \frac{\dif{}}{\dif{\varepsilon}} \int_\Sigma
                f \lbrack \hat{k}_1 (u,v), \hat{k}_2 (u,v),
                \hat{k}_{1,u} (u,v) , \hat{k}_{2,v} (u,v)
                \rbrack \dif{u} \wedge \dif{v} \,
                \right|_{\varepsilon = 0}\\
            \begin{aligned}
                & = \left.
                    \int_\Sigma \left(
                    \frac{\partial f}{\partial k_1} \frac{\partial \hat{k}_1}{\partial \varepsilon}
                    + \frac{\partial f}{\partial k_2} \frac{\partial \hat{k}_2}{\partial \varepsilon}
                    + \frac{\partial f}{\partial k_{1,u}} \frac{\partial \hat{k}_{1,u}}{\partial \varepsilon}
                    + \frac{\partial f}{\partial k_{2,v}} \frac{\partial \hat{k}_{2,v}}{\partial \varepsilon}
                    \right) \dif{u} \wedge \dif{v} \, \right|_{\varepsilon=0} \\
				& = \int_\Sigma \left(
                    \frac{\partial f}{\partial k_1}h_1
                    + \frac{\partial f}{\partial k_2}h_2
                    + \frac{\partial f}{\partial k_{1,u}}h_{1,u}
                    + \frac{\partial f}{\partial k_{2,v}}h_{2,v}
                    \right) \dif{u} \wedge \dif{v} \\
				& =\int_\Sigma \left(
                    \left(
                    \frac{\partial f}{\partial k_1}
                    - \frac{\partial}{\partial u} \frac{\partial f}{\partial k_{1,u}}
                    \right) h_1
                    + \left(
                    \frac{\partial f}{\partial k_2}
                    - \frac{\partial}{\partial v} \frac{\partial f}{\partial k_{2,v}}
                    \right)h_2
                    \right) \dif{u} \wedge \dif{v},
			\end{aligned}
    \end{multline*}
using the compact support condition. To find the critical points, we want $ \dif{L}/\dif{\varepsilon} \vert_{\varepsilon = 0} $ to be zero for arbitrary $ h_1 $ and $ h_2 $, so we have to show the stationary condition
        \begin{empheq}[left = {\empheqlbrace \,}]{alignat = 2}
					& \frac{\partial f}{\partial k_1} - \frac{\partial}{\partial u}\frac{\partial f}{\partial k_{1,u}} = 0, \label{EL1} \\
          & \frac{\partial f}{\partial k_2} - \frac{\partial}{\partial v}\frac{\partial f}{\partial k_{2,v}} = 0. \label{EL2}
				\end{empheq}
				These equations are called the \emph{Euler-Lagrange (differential) equations}. 
    
            Because
				\begin{align*}
					\frac{\partial f}{\partial k_1} & = -\frac{2k_{1,u}k_{2,v}}{\left(k_1 - k_2 \right)^3}, \\
					\frac{\partial f}{\partial k_2} & = \frac{2k_{1,u}k_{2,v}}{\left(k_1 - k_2 \right)^3}, \\
					\frac{\partial}{\partial u}\frac{\partial f}{\partial k_{1,u}} & = \frac{k_{2,uv}}{\left(k_1 - k_2 \right)^2} -
					\frac{2k_{2,v}\left(k_{1,u} - k_{2,u} \right)}{\left(k_1 - k_2 \right)^3}, \quad \text{and} \\
					\frac{\partial}{\partial v}\frac{\partial f}{\partial k_{2,v}} & = \frac{k_{1,uv}}{\left(k_1 - k_2 \right)^2} - \frac{2k_{1,u}\left(k_{1,v} - k_{2,v} \right)}{\left(k_1 - k_2 \right)^3},
				\end{align*}
				from \eqref{EL1} we have
				\begin{align}
					& 0 = -\frac{2k_{1,u}k_{2,v}}{\left(k_1 - k_2 \right)^3} - \frac{k_{2,uv}}{\left(k_1 - k_2 \right)^2} + \frac{2k_{2,v}\left(k_{1,u} - k_{2,u} \right)}{\left(k_1 - k_2 \right)^3} \notag\\
					& \therefore \quad \frac{(k_1-k_2)k_{2,uv}+2k_{2,u}k_{2,v}}{(k_1-k_2)^3} = 0.  \label{eq1}
				\end{align}
				Similarly, from \eqref{EL2} we have
				\begin{align}
					& 0 = \frac{2k_{1,u}k_{2,v}}{\left(k_1 - k_2 \right)^3} - \frac{k_{1,uv}}{\left(k_1 - k_2 \right)^2} + \frac{2k_{1,u}\left(k_{1,v} - k_{2,v} \right)}{\left(k_1 - k_2 \right)^3} \notag\\
					& \therefore \quad	\frac{(k_2-k_1)k_{1,uv}+2k_{1,u}k_{1,v}}{(k_1-k_2)^3} = 0. \label{eq2}
				\end{align}
				As we are assuming umbilic-free, the numerators of \eqref{eq1} and \eqref{eq2} lead to the desired conclusions \eqref{eq:lieMinimalConditions}.
\end{proof}

\begin{corollary}\label{Corollary: surface of revolution is Lie minimal}
	Every rotational surface in $M_\kappa$ is Lie minimal.
\end{corollary}

\begin{proof}
    As we have seen in Example~\ref{exp:rotationalSurfaces}, $k_{1,u} = k_{2,u}=0$ for rotational surfaces, hence, \eqref{eq:lieMinimalConditions} is satisfied. 
\end{proof}
\section{Lie minimal cmc surfaces}\label{sec:lieMinimalcmc}

In this section we consider the relationship between Lie minimal and cmc
$H$ (or minimal) surfaces. Recall that every cmc (or minimal) surface admits isothermic coordinates, which we will denote by $(u,v)$.

\begin{lemma}\label{Lemma: cmc k1k2}
	Let $X : \Sigma \rightarrow M_\kappa$ be Lie minimal with cmc $H$. Then, the principal curvatures of $X$ satisfy
    $$k_1-k_2=\alpha(u)\beta(v),$$ 
    where $\alpha$ (resp. $\beta$) is a function only depending on $u$ (resp. $v$).
\end{lemma}

\begin{proof}
	Suppose that $X$ is Lie minimal, and let $k_1-k_2>0$ without loss of generality. Since $X$ is Lie minimal, a simple calculation and Lemma~\ref{lem:lieMinimalConditions} show
    \begin{align*}
        (\log(k_1-k_2))_{uv}=\frac{H_u H_v}{(k_1 - k_2)^2},
    \end{align*}
    which vanishes because $H$ is constant. We conclude that
	\begin{align*}
		\log(k_1-k_2)&=\tilde{\alpha}(u)+\tilde{\beta}(v)\\
		\therefore \quad k_1-k_2&=\alpha(u)\beta(v),
	\end{align*}
	where $\alpha$ (resp. $\beta$) is a function only depending on $u$ (resp. $v$) and $e^{\tilde{\alpha}(u)}=\alpha(u)$ (resp. $e^{\tilde{\beta}(v)}=\beta(v)$).
\end{proof}

\begin{theorem}\label{Theorem: cmc Lie minimal}
    Let $X:\Sigma \to M_\kappa$ be a constant mean curvature $H$ surface. Then, $X$ is Lie minimal if and only if it is a rotational surface. 
\end{theorem}

\begin{proof}
	It follows from Corollary~\ref{Corollary: surface of revolution is Lie minimal} that every rotational surface is Lie minimal. 
 
    Suppose conversely that $X$ has cmc $H$ and is Lie minimal and is non-tubular. By solving the Codazzi equation \eqref{The Codazzi equation for isothermic coordinates}, we find
	\begin{equation*}
		\sigma = \frac{1}{2} \left(\log \gamma -\log(k_1-k_2)\right) \quad \therefore \quad e^{2\sigma}=\frac{\gamma}{k_1-k_2},
	\end{equation*}
	where $\gamma$ is some positive constant. Because of Lemma~\ref{Lemma: cmc k1k2}, we can write this as
	\[\sigma = \frac{1}{2}\left(\log \gamma-\big(\tilde{\alpha}+\tilde{\beta}\big)\right) \quad \text{and} \quad \quad e^{2\sigma}=\frac{\gamma}{\alpha \beta},\]
	and using the Gauss equation \eqref{The Gauss equation for isothermic coordinates}, we have
	\[-\frac{1}{2}\left(\tilde{\alpha}''+\tilde{\beta}''\right)+\left(\kappa+K\right)e^{2\sigma}=0. \]
	Now we differentiate both sides with respect to $u$ and $v$ and obtain
	\[0=\frac{\alpha'\beta'}{4\alpha^2\beta^2}(\kappa +K ).\]
	Thus there are two possibilities. On the one hand, if $\alpha'\beta'=0$, then without loss of generality, $\alpha$ is constant and $\sigma, k_1, k_2$ only depend on $v$. The surface is thus a channel linear Weingarten surface and by Proposition~\ref{proposition:channelWeingartenSurfacesAreRotational}, rotational.

	On the other hand, if $\kappa+K=0$, we have
	\[(k_1, k_2)=\left(H+\sqrt{H^2+\kappa},H-\sqrt{H^2+\kappa}\right),\]
	where we assumed $k_1-k_2>0$ without loss of generality.
    Since both principal curvatures are constant, the surface is tubular. 
\end{proof}

A complete classification of Lie minimal cmc surfaces is thus provided by explicit parametrizations of all rotational cmc surfaces (see for instance \cite{polly2023rotational}). 

\begin{corollary}
    Let $X:\Sigma \to \R^3$ be Lie minimal and minimal and not part of a plane. Then, $X$ is (part of) the catenoid. 
\end{corollary}

\begin{proof}
    If $X$ is not restricted to a plane it is rotational according to Theorem~\ref{Theorem: cmc Lie minimal}. Since the catenoid is the only rotational minimal surface in $\R^3$ this finishes the proof. 
\end{proof}
\section{Lie minimal Weingarten surfaces}\label{sec:lieMinimalWeingarten}
In this section we study Lie minimal linear and affine Weingarten surfaces in $\R^3$. Many of our results extend to other space forms, but in order to give a streamlined presentation we restrict ourselves to $\kappa=0$. 

First, we consider the class of linear Weingarten surfaces. For a surface $X:\Sigma\to\R^3$ satisfying \eqref{eq:LW}, we define $\Delta = b^2 - ac$. The surface $X^t = X+t N$, where $N$ denotes the Gauss map of $X$, is called the \emph{parallel surface at distance $t$}. It is well-known that obtaining parallel surfaces amounts to applying certain Lie sphere transformations \cite[Sec 4.4]{cecil}, often referred to as \emph{parallel transformations}.

\begin{proposition}[Bonnet's theorem {\cite[Section 3.4]{MR972503}}]\label{prop:Bonnet}
    Every surface parallel to a linear Weingarten surface is again linear Weingarten satisfying the linear Weingarten condition
    \[
        a^t K^t + 2b^t H^t + c^t = 0,
    \]
    with
    \[
        \Delta^t = (b^t)^2 - a^tc^t = \Delta.
    \]
    Every linear Weingarten surface with $\Delta>0$ is parallel to either
        \begin{itemize}
            \item a surface of positive constant Gauss curvature and hence a pair of cmc surfaces, or
            \item a minimal surface.
        \end{itemize}
\end{proposition}

As mentioned in Subsection~\ref{sec:lieMinimalSurfaces}, Lie minimality is preserved under Lie sphere transformations, and hence also preserved under parallel transformations.
These are exactly the transformations that map a surface to one of its parallel surfaces. Thus, we arrive at the following theorem. 

\begin{theorem}\label{thm:LieMinimalLW}
    Every non-tubular Lie minimal linear Weingarten surface in $\R^3$ with $\Delta>0$ is a surface of revolution. 
\end{theorem}

\begin{proof}
    Assume that $X:\Sigma\to\R^3$ is a Lie minimal linear Weingarten surface with $\Delta>0$. Then, there is a value $t_0$ such that the parallel surface $X^{t_0}$ is Lie minimal and cmc. According to Theorem~\ref{Theorem: cmc Lie minimal}, $X^{t_0}$ is rotational and thus has a rotational Gauss map $N$. Therefore, $X=X^{t_0} - t_0 N$ is rotational as well. 
\end{proof}

Finally, we consider the case of affine Weingarten surfaces. Recall, that a surface $X:\Sigma \to \R^3$ is called affine Weingarten if there are constants $x,y,y$ such that
\begin{equation}\label{Equation of affine linear Weingarten surface}
    xk_1+yk_2+z=0.
\end{equation}
We obtain the following theorem. 

\begin{theorem}\label{thm:AffineLieMinimal}
	Let $X : \Sigma \rightarrow \R^3$ be non-tubular affine Weingarten. Then it is Lie minimal if and only if it is a surface of revolution.
\end{theorem}

\begin{proof}
    Since every surface of revolution is Lie minimal (Corollary~\ref{Corollary: surface of revolution is Lie minimal}), let us assume that $X$ is Lie minimal and an affine Weingarten surface.
	From Lemma~\ref{lem:lieMinimalConditions}, we have
	\begin{align*}
		((x+y)k_1+z)k_{1,uv}-2xk_{1,u}k_{1,v}&=0 \\
        \textrm{and} \quad ((x+y)k_1+z)k_{1,uv}-2yk_{1,u}k_{1,v}&=0\\
		\therefore \quad (x-y)((x+y)k_1+z)k_{1,uv}&=0.
	\end{align*}
	There are three cases to consider: for $x-y=0$, the surface is Lie minimal cmc and thus rotational. 
	
    If $k_{1,uv}=0$, we have $k_{1,u}k_{1,v}=0$ because $X$ is Lie minimal and Lemma~\ref{lem:lieMinimalConditions}. Thus, $X$ is a channel surface and, according to Proposition~\ref{proposition:channelWeingartenSurfacesAreRotational}, a surface of revolution. 

    Finally, if $(x+y)k_1 + z=0$ both principal curvatures are constant. Thus, $X$ is tubular, which goes against our assumptions. 
\end{proof}
\section{Future work}
We have seen that linear Weingarten surfaces in $\R^3$ parallel to cmc surfaces are Lie minimal if and only if they are rotational in Theorem~\ref{thm:LieMinimalLW}. This result, of course, extends to the other ambient space forms. Since linear Weingarten surfaces belong to the realm of Lie sphere geometry, in the sense that they are those $\Omega$-surfaces that have constant conserved quantities (see \cite{burstall2012}) it seems a worthy goal to employ Lie sphere geometric methods to investigate Lie minimal linear Weingarten surfaces. This shall be the subject of a future project. 

\vspace{10pt}

\noindent\textbf{Acknowledgements.}
This paper is based on the fourth author's Master thesis, the goal of which was to investigate Lie minimal surfaces with additional curvature properties. 
The authors would like to express their gratitude to Professor Wayne Rossman for his support and continued interest in this project. Further, we thank Professor Udo Hertrich-Jeromin for fruitful discussions at Kobe University. This work 
was done while the third author was a JSPS International Research Fellow (Graduate 
School of Science, Kobe University) and has been supported by the JSPS 
Grant-in-Aid for JSPS Fellows 22F22701.

%\clearpage

\bibliographystyle{abbrvnat} % default
\DeclareUrlCommand\doi{\def\UrlLeft##1\UrlRight{doi:\href{http://dx.doi.org/##1}{##1}}\urlstyle{tt}}
%\providecommand{\natexlab}[1]{#1}
%\providecommand{\url}[1]{\texttt{#1}}
%\expandafter\ifx\csname urlstyle\endcsname\relax
%  \providecommand{\doi}[1]{doi: #1}\else
%  \providecommand{\doi}{doi: \begingroup \urlstyle{rm}\Url}\fi
% \bibliographystyle{ChemCommun} % minimum
%\nocite{*} % Show all references.

\bibliography{references}

\end{document}